\documentclass[12pt]{amsart}

\usepackage{latexsym, amssymb, amsmath}

\usepackage{amsfonts, graphicx}

\newcommand{\be}{\begin{equation}}
\newcommand{\ee}{\end{equation}}
\newcommand{\beq}{\begin{eqnarray}}
\newcommand{\eeq}{\end{eqnarray}}

\newtheorem{thm}{Theorem}[section]

\newtheorem{lma}{Lemma}[section]
\newtheorem{prop}{Proposition}[section]

\theoremstyle{remark}

\numberwithin{equation}{section}

\def\ep{\epsilon}

\def\p{\partial}
\def\S{\Sigma}

\def\tr{{\rm tr}}
\def\p{\partial}
\def\lf{\left}
\def\ri{\right}

\def\la{\langle}
\def\ra{\rangle}
\def\bg{\bar{g}}

\def\bg{\bar{g}}

\def\l{\lambda}

\def\Ric{\text{\rm Ric}}
\def\div{\text{\rm div}}

\def\vh{\vspace{.3cm}}
\def\Pi{\overline{\displaystyle{\mathbb{II}}}}

\def\a{\alpha}

\def\bD{\overline{\nabla}}
\def\bnabla{\overline{\nabla}}
\def\bnu{\overline{\nu}}
\def\th{ (\tr \hspace{.05cm} h)}
\def\divh{ \div \hspace{.05cm} h}
\def\vbg{d \mathrm{vol}_{\bg} }
\def\vsg{d \sigma_{\bg}}

\begin{document}

\title[]{Remarks on a scalar curvature rigidity theorem of Brendle and Marques}

\author{Graham Cox}
\address[Graham Cox]{Department of Mathematics, Duke University, Durham, NC 27708, USA.}
\email{ghcox@math.duke.edu}
\author{Pengzi Miao$^1$}
\address[Pengzi Miao]{School of Mathematical Sciences, Monash University, Victoria 3800, Australia;
Department of Mathematics, University of Miami, Coral Gables, FL 33124, USA.}
\email{Pengzi.Miao@sci.monash.edu.au; pengzim@math.miami.edu}
\author{Luen-Fai Tam$^2$}
\address[Luen-Fai Tam]{The Institute of Mathematical Sciences and Department of
 Mathematics, The Chinese University of Hong Kong,
Shatin, Hong Kong, China.}
\email{lftam@math.cuhk.edu.hk}
\thanks{$^1$Research partially supported by
Australian Research Council Discovery Grant  \#DP0987650
and by a 2011 Provost Research Award of the University of Miami}
\thanks{$^2$Research partially supported by Hong Kong RGC General Research Fund  \#CUHK 403011}
\renewcommand{\subjclassname}{%
  \textup{2010} Mathematics Subject Classification}
\date{}

\begin{abstract}
We give an improvement of a scalar curvature rigidity theorem of
Brendle and Marques regarding geodesic balls in $\mathbb{S}^n$.
The main result is that Brendle and Marques' theorem holds on a geodesic ball larger than that specified in \cite{BrendleMarques}.
\end{abstract}

\maketitle

\markboth{Graham Cox, Pengzi Miao and Luen-Fai Tam}{Remarks on a scalar curvature rigidity theorem}

\section{Introduction}

In a recent paper \cite{BrendleMarques}, Brendle and Marques proved the following theorem on scalar curvature rigidity of geodesic balls in the standard
$ n$-dimensional sphere $ \mathbb{S}^n$.

\begin{thm}[Brendle and Marques \cite{BrendleMarques}]    \label{thm-BM}
Let $\Omega=B(\delta) \subset \mathbb{S}^n $ be a closed geodesic ball of radius $\delta$ with
\be \label{eq-BM-rbd}
\cos\delta\ge \frac{2}{\sqrt{n+3}}.
\ee
 Let $ \bg $ be the
standard metric on $ \mathbb{S}^n$.
Suppose $g$ is another  metric on $\Omega$  with the properties:
\begin{itemize}
\item  $ R(g) \ge  R(\bg)$ at each point in $ \Omega$
\item  $ H(g) \ge H(\bg)  $ at each point on $ \partial \Omega$
\item  $ g $ and $ \bg$ induce the same metric on $ \p \Omega$
\end{itemize}
where $ R(g)$, $ R(\bg)$ are the scalar curvature of $ g$, $ \bg$, and
$H (g)$, $H(\bg)$ are the mean curvature of  $ \p \Omega $  in $(\Omega, g)$,
$(\Omega, \bg)$.
If $ g - \bg $ is sufficiently small in the $ C^2$-norm, then $\varphi^*(g)=\bg$
 for some diffeomorphism  $ \varphi: \Omega \rightarrow \Omega$
 such that  $\varphi|_{\p\Omega}=\text{\rm id}$.
\end{thm}

Theorem \ref{thm-BM} is an interesting rigidity result for domains in
$\mathbb{S}^n$ because the corresponding statement  is false for $\delta=\frac{\pi}{2}$, which
follows from the counterexample to  Min-Oo's conjecture (\cite{MinOo}) constructed by  Brendle,
Marques  and Neves in \cite{BrendleMarquesNeves}.
For an account of the connection of Theorem \ref{thm-BM} to other rigidity phenomena  involving scalar curvature, readers are referred to the recent survey  \cite{Brendle} by Brendle.

In this paper, we provide an improvement of Theorem \ref{thm-BM}
 by showing that Theorem \ref{thm-BM} is still valid on geodesic balls strictly {\em larger}  than those specified by \eqref{eq-BM-rbd}.
Precisely, we prove that condition  \eqref{eq-BM-rbd}  in Theorem \ref{thm-BM} 
can be  replaced by either one of the following  {weaker} conditions:

{\em \begin{itemize}
  \item[({\bf a})]  $ \cos \delta > \zeta$, where $ \zeta$ is the positive constant given by
  $$
 \zeta^2 =
 \frac{ 4 ( n+4) - 4 \sqrt{2n -1} }{ n^2 + 6 n + 17 } .
 $$
  \item[({\bf b})]  $\cos\delta>  \cos \delta_0$, where $\delta_0$ is the unique zero of the function
  $$
 F (\delta) = \alpha(\delta) +
\frac{ (n+3) \cos^2 \delta - 4 }{4  \sin^2 \delta }
 $$
 where
  $ \a(\delta)=   \frac {(  n + 1  )}{8 n}  \lf[1-   \lf( 1 - \frac{n}{2 \mu(\delta)} \ri) \cos\delta    \ri] ^{-1} $
  and $\mu(\delta)$ is the first nonzero Neumann eigenvalue of $B(\delta)$.  In particular,
  $\delta_0 $ satisfies
 \be  \label{eq-7n}
 (\cos\delta_0)^2 <\frac{7n-1}{2n^2+5n-1}.   
 \ee
  \end{itemize}}

We compare the conditions ({\bf a}) and ({\bf b}). It follows from \eqref{eq-7n}  that  $\delta_0 $ in  ({\bf b})  satisfies
\be \label{eq-compare1}
\limsup_{n\to\infty}\frac{(\cos\delta_0)^2}{\frac{4}{n+3}}\le \frac78 ,
\ee
while in  ({\bf a})  one has
\be \label{eq-compare3}
\lim_{n\to\infty}\frac{\frac{4(n+4)-4\sqrt{2n-1}}{n^2+6n+17}}{\frac{4}{n+3}}=1.
\ee
Therefore, ({\bf b}) gives a better improvement
of Theorem \ref{thm-BM} for large $n$.

For relatively small $n $,  ({\bf a})  appears to
be a better condition. For instance,  the constant $ \zeta$ in
({\bf a})  is given by
\be
 \zeta \approx
\left\{
\begin{array}{ll}
0.6581,  & n = 3 \\
0.6130,  & n = 4  \\
0.5774,  & n = 5,
\end{array}
\right.
\ee
while  $ \cos \delta_0 $ in ({\bf b}) is  restricted by (see by Lemma \ref{lma-alpha} (iii)),
\be
\cos \delta_0 > \kappa \approx
\left\{
\begin{array}{ll}
0.6919, & n = 3 \\
0.6512, & n = 4  \\
0.6155, & n = 5.
\end{array}
\right.
\ee
Thus, ({\bf a}) provides a better improvement of  Theorem \ref{thm-BM}
at least for dimensions $n = 3, 4, 5$.

\vspace{.2cm}

{\em Acknowledgment}.  The first author would like to thank Hubert Bray and Michael Eichmair for
 helpful discussions. The third author wants to thank Yuguang Shi for useful discussions.

\section{rigidity of geodesic balls}\label{s-rigidity}
Throughout this paper, we let $\Omega = B(\delta) \subset \mathbb{S}^n $ be a (closed) geodesic ball  
of radius $ \delta < \frac{\pi}{2}$, with boundary $ \Sigma = \p B (\delta) $. We denote by $ \bg $ the  standard metric on $ \mathbb{S}^n$, with volume form $ \vbg $ (\textit{resp.} $\vsg$) on $ \Omega$  (\textit{resp.} $\Sigma$). We additionally define $ \bD $ and $ \Delta_{\bg}$ to be the covariant  derivative and Laplace operator of $ \bg$,  and adopt the convention that the divergence, trace and norm
(denoted by $ \div (\cdot) $, $ \tr (\cdot) $ and $ | \cdot |$, respectively) are always computed with respect to $ \bg$.

We assume that $ g = \bg  + h $ is a metric close to $ \bg $ (say $ | h | \le \frac12 $ at each point in $ \Omega$) and  that $ g $ and $ \bg $ induce the same metric on $ \Sigma $. The outward unit normal to $ \Sigma $ in   $ (\Omega, \bg)$ is denoted by $ \bnu$, and $ X $ is the vector field on $ \Sigma$ dual to the $1$-form  $ h ( \cdot , \bnu) |_{T(\Sigma)}$, \textit{i.e.} $ \bg (v, X ) = h (v, \bnu) $ for any vector $ v $  tangent to $ \Sigma $. Finally, for any function $f$ and vector $\nu$, $ \p_\nu  f $ denotes the  directional derivative of $f$ along  $ \nu$.

\subsection{Brendle and Marques' proof}\label{s-review}
The following weighted integral estimate of  $(R(g) - R(\bg) )$ and $ ( H(g) - H(\bg) ) $  plays a key role in the proof of  Theorem \ref{thm-BM} in \cite{BrendleMarques}.

\begin{thm}[Brendle and Marques \cite{BrendleMarques}]  \label{thm-BM-est}
Let $ \Omega = B(\delta)$ and $ \l = \cos r $, where
$ r $ is the $ \bg $-distance to the center of $ B(\delta) $.
Assume  $ \div (h) = 0 $ where
$ h = g - \bg $. Then
\begin{equation*}\label{eq-RH rigidity 1}
 \begin{split}
  & \int_\Omega [R(g)- n (n-1) ] \l \ \vbg   +\int_\S (2- h(\bnu, \bnu) )[H(g)-H(\bg)]\l \ \vsg \\
    =   & \
  \int_\Omega \lf[   - \frac14 ( | \bD h |^2 + | \bnabla \th  |^2 )-\frac12\lf(|h|^2+\th^2\ri)  \ri] \l
   \ \vbg \\
& \ + \int_\Sigma H ({\bg}) \lf[ - \frac14  h(\bnu, \bnu)^2   - \frac n{2(n-1)}  |X|^2   \ri] \l \ \vsg  \\
& \ + \int_\Sigma  \lf[ - h(\bnu, \bnu)^2  - \frac12 | X|^2 \ri] \p_{\bnu} \l  \ \vsg
  + \int_\Omega E(h)  \ \vbg +  \int_\Sigma F(h) \ \vsg
 \end{split}
 \end{equation*}
where
$
| E(h) | \le C ( | h|^3 + | \bD h |^3 ), \
|F(h)|\le C  \lf(|h|^3+|h|^2|\bD h|\ri)
$
for some constant $C$ depending only on $n$.
\end{thm}

To see how Theorem \ref{thm-BM} follows from Theorem \ref{thm-BM-est}, one first  pulls back $ g$ through a diffeomorphism $ \varphi $: $ \Omega \rightarrow \Omega$ with $ \varphi |_\Sigma = \text{\rm id} $
such that $ \varphi^*(g) - \bg $ is $ \bg $-divergence free and $ || \varphi^*(g) - \bg ||_{W^{2,p} (\Omega) }
\le N || g - \bg ||_{W^{2,p} (\Omega) } $  for some $ p>n$ and $ N$ depending only on $ \Omega$  (\cite[Proposition 11]{BrendleMarques}). Replacing $ g $ by $ \varphi^*(g) $, one assumes that $\div (h) = 0$,
where $ h = g - \bg $  and $ || h ||_{W^{2, p} (\Omega) } $ is small. If $ R(g) \ge n (n-1)$ and $ H(g) \ge H(\bg)$, Theorem \ref{thm-BM-est} then implies
\be \label{eq-BM-est-0}
\begin{split}
& \ \int_\Omega \lf[   \frac14 ( | \bD h |^2 + | \bnabla \th  |^2 ) + \frac12\lf(|h|^2+\th^2\ri)  \ri] \l
   \ \vbg \\
& \ + \int_\Sigma h(\bnu, \bnu)^2 \lf[ \frac14 H(\bg) \l + \p_{\bnu} \l \ri]
+  |X|^2  \lf[ \frac n{2(n-1)} H ({\bg}) \l  + \frac12 \p_{\bnu} \l  \ri]   \vsg  \\
\le & \   C || h ||_{C^1(\bar{\Omega})} \int_\Omega  \lf( | \bD h |^2 + | h |^2 \ri)  \vbg
\end{split}
\ee
for a constant  $ C $ independent on $ h$. At $ \Sigma$, direct calculation shows
\be \label{eq-BM-est-1}
 \frac14 H(\bg ) \l + \p_{\bnu} \l = \frac{(n+3) \cos^2\delta - 4 }{4 \sin \delta }
\ee
\be \label{eq-BM-est-2}
\frac n{2(n-1)} H ({\bg}) \l  + \frac12 \p_{\bnu} \l  = \frac{ ( n + 1 ) \cos^2 \delta - 1 }{ 2 \sin \delta } .
\ee
If
$ \cos \delta \ge { \frac{ 2}{\sqrt{n+3}} } $, then both quantities in \eqref{eq-BM-est-1} and \eqref{eq-BM-est-2}
are nonnegative. Therefore, \eqref{eq-BM-est-0} implies $ h = 0 $ if $ || h ||_{C^1( \bar{\Omega} ) } $
is sufficiently small.

\subsection{Improvement of Theorem \ref{thm-BM}: approach 1}\label{s-method 1}

Let $ \l $ and $ h  $ be given as in Theorem \ref{thm-BM-est}. Define
\be \label{eq-defofwh}
\begin{split}
W (h) = & \ \int_\Omega \lf[   \frac14 ( | \bD h |^2 + | \bnabla \th  |^2 ) + \frac12\lf(|h|^2+\th^2\ri)  \ri] \l
   \ \vbg \\
& \ + \int_\Sigma h(\bnu, \bnu)^2 \lf[ \frac14 H(\bg) \l + \p_{\bnu} \l \ri]
 + |X|^2  \lf[ \frac n{2(n-1)} H ({\bg}) \l  + \frac12 \p_{\bnu} \l  \ri]   \vsg  .
 \end{split}
\ee
It is clear from the above Brendle and Marques' proof
that Theorem \ref{thm-BM} holds  on  a geodesic ball $ \Omega = B(\delta)$
provided one can prove
\be \label{eq-BM-est-hope}
W (h) \ge    \ep \int_\Omega  \lf( | \bD h |^2 + | h |^2 \ri)  \vbg
\ee
for some positive $ \epsilon $ independent on $h$.
To show \eqref{eq-BM-est-hope},
the difficulty lies in handling the boundary integral
$$ \int_\Sigma h(\bnu, \bnu)^2 \lf[ \frac14 H(\bg) \l + \p_{\bnu} \l \ri]
+ |X|^2  \lf[ \frac n{2(n-1)} H ({\bg}) \l  + \frac12 \p_{\bnu} \l  \ri]   \vsg
$$
which can be  negative if $ \cos \delta$ is small.

\begin{prop} \label{prop-bdry-1}
Let $h$ be any $C^2$ symmetric (0,2) tensor on $\Omega = B(\delta)$
with $ \div (h) = 0 $. Let $ c = \cos \delta $ and $ s = \sin \delta$.
 Given any positive  function $w$ on $\Omega$, we have
 \begin{equation}
 \label{eq-trh-est-3}
s\int_\S \th h (\bnu,\bnu)\vsg\le \int_\Omega\lf[  \frac w2\sqrt{1-\l^2}|h|^2+\l \th^2
+\frac1{2w}\sqrt{1-\l^2}|\bD\th|^2\ri]\vbg.
\end{equation}
In particular, if $h|_{T(\S)}=0$, then
 \begin{equation}
 \label{eq-trh-est-4}
s\int_\S h(\bnu, \bnu)^2 \vsg \le  \int_\Omega\lf[  \frac w2\sqrt{1-\l^2}|h|^2+\l \th^2
 +\frac1{2w}\sqrt{1-\l^2}|\bD\th|^2\ri]\vbg.
\end{equation}
\end{prop}

\begin{proof}
Let $\omega$ be the $1$-form on $ \Omega$ given by
$$ \omega_k = \th h_{ik} \bD^i \l  .$$
Using the fact $\bD_k\bD^i\l=-\l\delta_k^i$ and the assumption  $\div (h)=0$, we have
$$
\bD^k \omega_k = - \l \th^2 + h(\bD\l,\bD \th).
$$
At $\S$,  $\omega(\bnu)=-s \th h(\bnu,\bnu)$. It follows from the divergence theorem
\be \label{eq-beforeholder}
s\int_\S \th  h(\bnu,\bnu)\vsg=\int_\Omega \lf[\l \th^2- h(\bD\l,\bD \th)\ri]\vbg.
\ee
Given any positive function $w$ on $ \Omega$, using the fact $ | \bD \l  |^2 = 1 - \l^2 $, we have
\begin{equation} \label{eq-holder}
\begin{split}
- h(\bD\l,\bD \th)\le &|\bD\l|\,|h|\, |\bD \th |\\
\le &\sqrt{1-\l^2}\lf[\frac w2|h|^2+\frac1{2w}|\bD\th |^2\ri].
\end{split}
\end{equation}
Thus, \eqref{eq-trh-est-3} follows from \eqref{eq-beforeholder} and \eqref{eq-holder}. If $h|_{T(\S)}=0$, 
$h(\bnu,\bnu)=  \tr \hspace{.05cm} h $ at $ \Sigma$. Therefore, \eqref{eq-trh-est-3} implies \eqref{eq-trh-est-4}.
\end{proof}

\begin{thm}\label{thm-method-1}
Let $ \delta $ be a constant in $(0, \frac{\pi}{2})$.
 Suppose $ \cos \delta > \zeta$, where $ \zeta $ is the positive 
 constant given by 
 \be \label{eq-thmii-condition}
 \zeta^2 =
 \lf\{
 \begin{array}{lcl}
 \frac{2}{n+1} & \  \mathrm{if} \ n \le 4 \\
 \frac{ 4 ( n+4) - 4 \sqrt{2n -1} }{ n^2 + 6 n + 17 } & \ \ \mathrm{if} \ n \ge 5  .
 \end{array}
 \ri.
 \ee
 Then the conclusion of Theorem \ref{thm-BM} holds on $ B(\delta)$.
\end{thm}

\begin{proof}
Let $c=\cos\delta$. Note that \eqref{eq-thmii-condition} implies $ c^2 \ge \frac{1}{{n+1}} $, hence the coefficient
of $ | X|^2 $ in \eqref{eq-defofwh}  is nonnegative.
 By Theorem \ref{thm-BM}, it suffices to assume
$ c^2 < \frac{4}{n+3} $.
Choosing $ w = \sqrt{2} $ in Proposition \ref{prop-bdry-1}, we have
\begin{equation}
\label{eq-BM-est-method1-2-2}
\begin{split}
W(h) \ge & \ \int_\Omega \lf[   \frac14 ( | \bD h |^2 + | \bnabla \th  |^2 ) + \frac12 \lf(|h|^2+\th^2\ri)  \ri] \l
   \ \vbg \\
& \ + \frac{ (n+3)c^2 - 4 }{4(1-c^2)}   \sqrt{2 (1-c^2) }     \int_\Omega
\lf(  \frac12 |h|^2 + \frac14 | \bD \th |^2 \ri) \vbg \\
&  + \frac{ (n+3)c^2 - 4 }{4(1-c^2)}  \int_\Omega \l \th^2 \vbg .
\end{split}
\end{equation}
We seek conditions on $ c$ such that
\be \label{eq-lmaii-1}
c + \frac{ (n+3)c^2 - 4 }{4(1-c^2)}   \sqrt{2 (1-c^2) }  > 0
\ee
and
\be \label{eq-lmaii-2}
\frac12 +  \frac{ (n+3)c^2 - 4 }{4(1-c^2)}  \ge 0  .
\ee
Direct calculation shows that \eqref{eq-lmaii-1} (under the assumption $ c^2 < \frac{4}{n+3} $) is equivalent to
\be
c^2 > \frac{ 4 ( n+4) - 4 \sqrt{2n -1} }{ n^2 + 6 n + 17 }
\ee
and
 \eqref{eq-lmaii-2} is equivalent to
\be
c^2 \ge \frac{2}{n+1} .
\ee
Since
\be \label{eq-nge5}
\frac{ 4 ( n+4) - 4 \sqrt{2n -1} }{ n^2 + 6 n + 17 }  \ge \frac{2}{n+1}
\ee
precisely when $ n \ge 5$,  we conclude that
\eqref{eq-BM-est-hope} holds for some $ \ep > 0$ if \eqref{eq-thmii-condition} is satisfied.  Theorem \ref{thm-method-1} is proved.
\end{proof}

Theorem \ref{thm-method-1}  verifies condition ({\bf a}) in the introduction
for $ n \ge 5$. The remaining case $n = 3, 4$ in condition ({\bf a}) will be verified
in section \ref{sec-combined}.

\subsection{Improvement of Theorem \ref{thm-BM}: approach 2}\label{s-method 2}
In this section, we give a different approach to  estimate the boundary integral  of
$ \th^2  $ in  $W(h)$ in terms of the interior integral in $W(h)$. To do so,  we use the linearization of the scalar curvature \eqref{eq-reduced-linear-R}. Noticing that the integral of $\tr \hspace{.5mm} h$ over $B(\delta)$ is close to zero, we  apply the Poincar\'e inequality  through an estimate of the first nonzero Neumann eigenvalue of $B(\delta)$ in \cite{MiaoTam}.

\begin{lma}\label{lma-trace-est-1}
Let $\Omega \subset \mathbb{S}^n $ be a closed domain  with smooth boundary $\S$.
Let $ \bg $ be the standard metric on $ \mathbb{S}^n$ and
 $g=\bg+h$ be another  smooth metric on  $\Omega $
such that $g$, $\bg$ induce the same metric on $\S$ and $\divh=0$. Suppose $ | h | $ is very small, say
$ | h | \le \frac12 $ at every point.
 \begin{enumerate}
             \item [(i)]  Given any  smooth function $ f $ on $ \Omega$, one has
             \begin{equation*}
             \begin{split}
&   \int_\Omega  f  \th \Delta_{\bg} \th+(n-1) f   \th^2 \ \vbg \\
  =&\int_\Omega f  \th \lf[ R(\bg)-R(g) \ri] \vbg+ E(h, f)
  \end{split}
  \end{equation*}
  where
  $$ | E (h, f) | \le C || f ||_{C^1 ({\overline{\Omega}} )}  \lf( \int_\Omega \lf(|h|^3+|\bD h|^3\ri) \vbg   + \int_\Sigma   |h|^2 |\bD h| \vsg \ri) $$
for a positive constant $C$ depending only on $(\Omega,\bg)$.

\vh
             \item [(ii)]
    \begin{equation*}
             \begin{split}
  \int_\Omega \th \vbg=& - \frac1{n-1}  \lf( \int_\Omega  \lf[ R(g)-R(\bg) \ri]  \vbg \ri. \\
& \lf.   +  2 \int_\S \lf[ H(g)-H(\bg) \ri]   \vsg \ri)
  +F(h)
  \end{split}
  \end{equation*}
  where
  $$ | F(h) | \le C \lf( \int_\Omega \lf(|h|^2+|\bD h|^2\ri) \vbg
  + \int_\Sigma  ( |h|^2 + |h |  |\bD h| ) \vsg   \ri)$$
       for a positive constant $C$ depending only on $(\Omega,\bg)$.
             \end{enumerate}
\end{lma}

\begin{proof} Since $ \div (h) = 0 $ and $ \Ric(\bg) = (n-1) \bg$, $h$ satisfies
 \be \label{eq-reduced-linear-R}
 - \Delta_{\bg}  \th - ( n-1) \th  =   DR_{\bg} (h) ,
 \ee
 where $ DR_{\bg} (\cdot) $ denotes the linearization of the scalar curvature at $ \bg$.
By \cite[Proposition 4]{BrendleMarques} (also see  \cite[Lemma 2.1]{MiaoTam}),
one knows
\begin{equation} \label{eq-linearR}
 \begin{split}
 R(g) - R (\bg) & = DR_{\bg}(h)-\frac12DR_{\bg}(h^2)+\la h,\bD^2\th\ra \\
&   - \frac14 \lf(| \bD h |^2 +  | \bD (\tr_{\bg} h) |^2\ri)+\frac12h^{ij}h^{kl}\overline R_{ikjl} \\
& +E(h) +\bD_i(E_1^i(h))
 \end{split}
 \end{equation}
 where $ E(h) $ is a function and $ E_1(h)$ is a vector field on $ \Omega$
 satisfying
 $$ | E (h) | \le C ( | h| | \bD h |^2 + | h|^3 ), \ \
  | E_1 (h) | \le C  |h|^2 |\bD h|   $$
for a  positive constant $ C$ depending only on $ n$.
Multiplying  \eqref{eq-reduced-linear-R} by $ f \th$ and
integrating by parts, (i) follows from \eqref{eq-linearR}.

To prove {(ii)}, we  integrate  \eqref{eq-reduced-linear-R} on $ \Omega$ to get
\be
- (n-1) \int_\Omega \th \vbg = \int_\Omega DR_{\bg} (h) \vbg
+ \int_\Sigma \p_{\bnu} \th \ \vsg .
\ee
Let $ DH_{\bg} (h) $ denote the linearization of the mean curvature
of $ \Sigma$ at $ \bg$. Direct calculation (see \cite[Proposition 5]{BrendleMarques} or  \cite[(34)]{MiaoTam2009})
shows
\be \label{eq-dh}
2 DH_{\bg} (h)  = \p_{\bnu} \th - \divh(\bnu) - \div_\Sigma X .
\ee
Since  $ \div (h) = 0 $, \eqref{eq-dh} implies
\be
\int_\Sigma \p_{\bnu} \th \ \vsg = 2 \int_\Sigma DH_{\bg} (h) \vsg .
\ee
By \cite[Proposition 5]{BrendleMarques}, one has
\be \label{eq-linearH}
|H(g)-H(\bg)-DH_{\bg}(h)|\le C(|h|^2+|h||\bD h|)
\ee
for a positive constant $C$ depending only on $n$. (ii) now follows from \eqref{eq-linearR}-\eqref{eq-linearH}
and integration by parts on $ \Omega$.
\end{proof}

We will make use of the first nonzero Neumann eigenvalue of
$B(\delta) $, which we denote by $ \mu (\delta)$.
The next lemma on $ \mu (\delta) $ was proved in \cite[Lemma 3.1]{MiaoTam}.

\begin{lma}[\cite{MiaoTam}] \label{lma-mu-estimate}
Let $ \mu (\delta) $ be the first nonzero Neumann eigenvalue of $ B(\delta)$
 (with respect to $\bg$). Then
\begin{enumerate}
\item[(i)] $ \mu( \delta) $ is a strictly decreasing function of $ \delta$ on $(0, \frac{\pi}{2}]$;
\item[(ii)] for any $ 0< \delta< \frac{\pi}{2}$,
$$ \mu (\delta) > n + \frac{ (\sin \delta)^{n-2} \cos \delta}{ \int_0^\delta ( \sin  t)^{n-1} d t }
>  \frac{n}{ (\sin \delta)^{2} }.$$
\end{enumerate}
\end{lma}

Using $ \mu(\delta)$, we have the following estimate of $\int_\Sigma \th^2 \vsg$.

 \begin{prop}\label{prop-bdry-2} Let $\Omega = B(\delta) $ and
  $ \mu (\delta) $ be the first nonzero Neumann eigenvalue of $B(\delta) $.
 Let $g=\bg+h$ be a smooth metric on  $ B(\delta)  $
such that $g$, $\bg$ induce the same metric on $\S$ and $\div (h) =0$. Suppose $ | h | $ is  small, say
$ | h | \le \frac12 $ at every point. Let $ c = \cos \delta $ and $ s = \sin \delta$. Then
 \begin{equation*}
\begin{split}
s \int_\Sigma  \th^2 \vsg \le &  \ 2  \lf[1-   c \lf( 1 - \frac{n}{2 \mu(\delta)} \ri)
\ri]   \int_\Omega \l |\bD \th |^2 \ \vbg \\
& \ - 2\int_\Omega ( \l - c ) \th (R(g )-R(\bg))\ \vbg\\
 & \ +C ||h||_{C^1}\lf[ \int_\Omega \lf(|h|^2+|\bD h|^2\ri)   \vbg
+  \int_\Sigma |h|^2   \ \vsg\ri]\\
 & \ +C  \lf[\int_\Omega (R(g)-R(\bg))\ \vbg+2\int_\S (H(g)-H(\bg))\ \vsg\ri]^2
\end{split}
\end{equation*}
 for some positive constant $C $ depending only on $ (\Omega, \bg)$ and $c$.
 \end{prop}

 \begin{proof}
Integrating by parts,  using the fact $ \l = c $ at $ \Sigma$ and $ \Delta_{\bg} \l = - n \l $
on $ \Omega$, we have
\be\label{eq-RH rigidity 2}
\begin{split}
\int_\Sigma   \th^2 \p_{\bnu} \l \ \vsg
 = & \int_\Omega  \th^2 \Delta_{\bg} \l   - (\l - c)  \Delta_{\bg} \th^2  \ \vbg \\
 = & \int_\Omega  - n \l \th^2    - 2 ( \l - c ) [ \th \Delta_{\bg} \th + | \bD \th |^2 ]   \vbg .
\end{split}
\ee
Choosing $ f = \l - c $ in Lemma \ref{lma-trace-est-1}(i), we have
  \begin{equation} \label{eq-fandl}
  \begin{split}
&  \int_\Omega  (\l - c)  \th \Delta_{\bg} \th \ \vbg \\
= & \ \int_\Omega - (n-1) (\l - c)   \th^2  -  ( \l -c )  \th \lf[ R(g)-R(\bg) \ri] \vbg  +  E_2 (h)
  \end{split}
  \end{equation}
  where
 $$ | E_2(h) | \le C  \lf( \int_\Omega \lf(|h|^3+|\bD h|^3\ri) \vbg
 + \int_\Sigma |h|^2 |\bD h|  \vsg \ri) $$
    for some constant $C$ depending  on $(\Omega,\bg)$ and $ c$.
It follows from \eqref{eq-RH rigidity 2} and  \eqref{eq-fandl}  that
\be\label{eq-RH rigidity 2-2}
\begin{split}
\int_\Sigma   \th^2 \p_{\bnu} \l \ \vsg
 = & \int_\Omega  \lf[ ( n - 2 )   \th^2    - 2     | \bD \th |^2 \ri] \l \  \vbg \\
 &  + 2 c  \int_\Omega  \lf[ | \bD \th |^2 - (n-1)     \th^2 \ri] \vbg  \\
 & +  2 \int_\Omega ( \l -c )  \th \lf[ R(g)-R(\bg) \ri] \vbg   -  2 E_2 (h) .
\end{split}
\ee
Since $ \l \ge c $ on $ \Omega$, \eqref{eq-RH rigidity 2-2} implies
\begin{equation*}
\begin{split}
\int_\Sigma   \th^2 \p_{\bnu} \l \ \vsg
 \ge &    - 2    \int_\Omega    | \bD \th |^2 \l   \vbg
  + 2 c  \int_\Omega  \lf[ | \bD \th |^2 - \frac n 2     \th^2 \ri] \vbg  \\
 & +  2 \int_\Omega ( \l -c )  \th \lf[ R(g)-R(\bg) \ri] \vbg   -  2 E_2 (h) .
\end{split}
\end{equation*}
By the variational characterization of $ \mu (\delta) $, we have
\be \label{eq-mu-new}
\int_\Omega | \bD  \th  |^2 \ \vbg \ge \mu(\delta)
\lf[ \lf( \int_\Omega \th^2 \ \vbg \ri) - \frac{1}{V(\bg)} \lf( \int_\Omega \th \ \vbg \ri)^2 \ri]
\ee
where $ V(\bg) = \int_\Omega 1 \vbg $.
It follows from  Lemma \ref{lma-trace-est-1}(ii) and \eqref{eq-mu-new} that
\be\label{eq-RH rigidity 3}
\begin{split}
&\int_\Omega \lf[  |\bD \th |^2 - \frac n 2   \th^2 \ri] \vbg \\
\ge & \lf( 1 - \frac{n}{2 \mu(\delta)} \ri)
\int_\Omega |\bD \th |^2\ \vbg \\
&  - C \lf[\int_\Omega (R(g)-R(\bg))\ \vbg+
2 \int_\S (H(g)-H(\bg))\ \vsg\ri]^2\\
&-  C \lf[ \int_\Omega \lf(|h|^2+|\bD h|^2\ri) \vbg
  + \int_\Sigma  ( |h|^2 + | h | |\bD h| \vsg  )\ri]^2
  \end{split}
  \ee
 for a positive constant $ C$ depending only on $(\Omega, \bg)$.
 The lemma now follows from \eqref{eq-RH rigidity 2-2}, \eqref{eq-RH rigidity 3}
 and the fact $ \l \le 1 $.
 \end{proof}

The following lemma is needed for the statement of Theorem \ref{thm-RH rigidity 1}.

\begin{lma}\label{lma-alpha}
On $(0, \frac{\pi}{2}]$, define
$$
\a(\delta)=  \lf[1-   \lf( 1 - \frac{n}{2 \mu(\delta)} \ri) \cos\delta    \ri] ^{-1}
 \frac {(  n + 1  )}{8 n}
 $$
 and
 $$
 F (\delta) = \alpha(\delta) +
\frac{ (n+3) \cos^2 \delta - 4 }{4  \sin^2 \delta }.
 $$
 Then
 \begin{enumerate}
   \item [(i)] $\a(\delta)$ is strictly decreasing, $\lim_{\delta \to 0+}\a(\delta)=\infty$ and $\a(\frac\pi2)=\frac{n+1}{8n}.$
   \item [(ii)]
 $F(\delta)$ is strictly decreasing,
$ \lim_{\delta \to 0+} F (\delta)=\infty$ and
 $F(\frac\pi2) < 0$. Hence there is exactly one
 $\delta_0\in (0,\frac\pi2)$ such that $F(\delta_0)=0$.

\item[(iii)]  $ \cos\delta_0 > \kappa $ where $ \kappa $ is the positive root of
the equation
 $$  2n(n+3)x^2+(n+1)x+(1-7n) = 0. $$
   In particular, $(\cos\delta_0)^2>\frac{1}{n+1}.$
 \end{enumerate}
\end{lma}
\begin{proof} (i) follows  directly from  Lemma \ref{lma-mu-estimate}.
(ii) follows from (i) and the fact 
$$
F(\delta) = \alpha(\delta) +  \frac{n-1}{4} \frac{1}{\sin^2 \delta} - \frac{n+3}{4} .
$$
To prove (iii),  suppose  $\cos\delta_0=a$. Since $0<1-\frac{n}{2 \mu(\delta_0)}<1$,  one has $
\lf(1-\frac{n}{2 \mu(\delta_0)}\ri)\cos\delta_0<a
$
and
$
\a(\delta_0)<\frac{n+1}{8n}\frac{1}{(1-a)}.
$
Therefore,
$$ 0  = F (\delta_0)  < \frac{n+1}{8n}\frac{1}{(1-a)} +
 \frac{n-1}{4} \frac{1}{1 - a^2} - \frac{n+3}{4} $$
which  implies (iii).
\end{proof}

\begin{thm}\label{thm-RH rigidity 1}
 Let $\Omega = B (\delta)$   be a geodesic ball of radius $\delta$ in $  \mathbb{S}^n$.
 Suppose $ \delta<\delta_0 ,$
where $ \delta_0 $ is the unique zero in $(0, \frac{\pi}{2})$ of the function
$$
F (\delta) = \alpha (\delta)  +
\frac{ (n+3) \cos^2 \delta - 4 }{4  \sin^2 \delta }
$$
where $ \alpha (\delta) = \lf[1-   \lf( 1 - \frac{n}{2 \mu(\delta)} \ri) \cos\delta    \ri] ^{-1}
 \frac {(  n + 1  )}{8 n}  $.
 Then the conclusion of Theorem \ref{thm-BM} holds on $\Omega $.
 \end{thm}

\begin{proof}
 Let $ W(h) $ be given in \eqref{eq-defofwh}. Let $ c  = \cos \delta$. 
Lemma \ref{lma-alpha}(iii) shows
$ c^2 > \frac{1}{ n+1 }$. Hence, the coefficient
of $ | X|^2 $ in $W(h)$  is nonnegative.
By Theorem \ref{thm-BM}, it suffices to assume  
 $ c^2  < \frac{4}{ n+3 }$.
 Apply
 Proposition \ref{prop-bdry-2}, we  have
\begin{equation}
\label{eq-BM-est-method2-2}
\begin{split}
W(h) \ge & \ \int_\Omega \lf[   \frac14 ( | \bD h |^2 + | \bnabla \th  |^2 ) + \frac12 \lf(|h|^2+\th^2\ri)  \ri] \l
   \ \vbg \\
& \ + \lf[\frac{ (n+3)c^2 - 4 }{4(1-c^2)}\ri] 2  \lf[1-   c \lf( 1 - \frac{n}{2 \mu(\delta)} \ri) \ri]   \int_\Omega  |\bD \th |^2 \l \ \vbg  \\
& + \hat{E}(h,c),
\end{split}
\end{equation}
where
\begin{equation} \label{eq-error2}
\begin{split}
\hat{E}(h,c) = & \lf[\frac{ (n+3)c^2 - 4 }{4(1-c^2)}\ri]
\lf\{  - 2\int_\Omega ( \l - c ) \th (R(g )-R(\bg)) \vbg   \ri. \\
 & +C ||h||_{C^1}\lf[ \int_\Omega \lf(|h|^2+|\bD h|^2\ri)   \vbg
+  \int_\Sigma |h|^2   \ \vsg\ri]\\
 & \lf. + C  \lf[\int_\Omega (R(g)-R(\bg))\ \vbg+2\int_\S (H(g)-H(\bg))\ \vsg\ri]^2 \ri\}.
\end{split}
\end{equation}
Since $ \delta < \delta_0 $,  Lemma \ref{lma-alpha} (ii) implies
$$
F(\delta) =\alpha(\delta) +
\frac{ (n+3) \cos^2 \delta - 4 }{4(1- \cos^2 \delta )} > F(\delta_0) = 0 .
$$
Hence there exists a small constant  $ \ep \in (0,1) $ such that
\be \label{eq-smallep}
\frac{1}{4} \lf( 1 + \frac {(1-  \ep)}{ n} \ri) + \lf[\frac{ (n+3)c^2 - 4 }{4(1-c^2)}\ri] 2  \lf[1-   c \lf( 1 - \frac{n}{2 \mu(\delta)} \ri) \ri]   > 0 .
\ee
By \eqref{eq-BM-est-method2-2} and \eqref{eq-smallep}, using the fact
 $ | \bD h |^2 \ge \frac1n | \bD \th |^2 $, we  have
\begin{equation}
\label{eq-BM-est-method2-3}
\begin{split}
W(h) \ge & \  \frac{1}{4}\ep  c \int_\Omega ( | \bD h |^2 + |h|^2  )   \ \vbg + \hat{E}(h,c).
\end{split}
\end{equation}

Now suppose
$R(g)-R(\bg)\ge0$,  $H(g)-H(\bg)\ge0$
 and $ || h ||_{W^{2,p}(\Omega)}$ is sufficiently small.
 It follows from Theorem \ref{thm-BM-est}, \eqref{eq-error2} and  \eqref{eq-BM-est-method2-3} that
\begin{equation}\label{eq-RH rigidity 5}
 \begin{split}
  & \frac12  \int_\Omega [R(g)-R(\bg)] \l \ \vbg   + \frac12 \int_\S [H(g)-H(\bg)]\l\ \vsg \\
    \le    & \ \ep \int_\Omega ( | \bD h |^2 + |h|^2  )   \ \vbg  \\
 &+C||h||_{C^1}\lf[ \int_\Omega \lf(|h|^2+|\bD h|^2\ri)   \vbg +  \int_\Sigma |h|^2   \ \vsg\ri].
 \end{split}
 \end{equation}
for some positive constant $ C$ independent of $h$.
We can then proceed as in \cite{BrendleMarques}:
since
$ ||h||_{L^2 (\Sigma)} \le C||h||_{W^{1,2}(\Omega) } $,
one knows  the  terms in the last line in \eqref{eq-RH rigidity 5}
is bounded by $ C|| h ||_{C^1(\overline{\Omega})} || h ||_{W^{1,2}(\Omega)} $.
Therefore,  if $ || h ||_{W^{2,p}({\Omega})}$ is sufficiently small,
\eqref{eq-RH rigidity 5} implies  $ h $ must  vanish
identically.  This completes the proof of Theorem \ref{thm-RH rigidity 1}.
\end{proof}

We give some lower estimates of $\delta_0$ which are relatively more explicit.

\begin{prop}\label{prop-cosd-estimate}
$ \delta_0 $ in Theorem \ref{thm-RH rigidity 1} satisfies
\begin{enumerate}

   \item [(i)]  $\delta_0>\tilde\delta_0$ where $\tilde\delta_0$ is the unique zero in $(0,\frac \pi 2)$ of the equation
       $$
\lf[1-\lf(1-\frac{n}{2 \tilde \mu (\delta)}\ri)
       \cos\delta\ri]^{-1}\frac{n+1}{8n} +
\frac{ (n+3) \cos^2 \delta - 4 }{4(1- \cos^2 \delta )} = 0
 $$
    where
     $  \displaystyle  \tilde \mu (\delta)=n+\frac{(\sin\delta)^{n-2}\cos\delta}{\int_0^\delta (\sin t)^{n-1}dt}.
       $
       
       \vh
       
       \item[(ii)] $ \cos \delta_0 <  \tilde \kappa $ where $\tilde \kappa $ is the unique zero in $(0,1)$ of the equation
           $$ n (n+3) x^4 +  n (n+3) x^3 +   2 n ( n+1) x^2  + (1 - 3n) x  - 7n + 1  = 0 . $$

\vh

       \item[(iii)]  $ \displaystyle  (\cos\delta_0)^2 <\frac{7n-1}{2n^2+5n-1} . $
 \end{enumerate}
\end{prop}

\begin{proof}
By Lemma \ref{lma-mu-estimate} (ii),   $\mu(\delta_0)>\tilde \mu (\delta_0)$. Hence,
\be \label{eq-compare-1}
\lf[1-\lf(1-\frac{n}{2 \tilde \mu (\delta_0)}\ri)
       \cos\delta_0 \ri]^{-1}\frac{n+1}{8n} +
\frac{ (n+3) \cos^2 \delta_0 - 4 }{4(1- \cos^2 \delta_0 )} < 0.
\ee
Note that $ \tilde \mu (\delta)$ is  strictly decreasing
in $(0, \frac{\pi}{2} ]$.
 As in the proof of Lemma \ref{lma-alpha}(ii),  we know the function
$$
\lf[1-\lf(1-\frac{n}{2 \tilde \mu (\delta)}\ri)
       \cos\delta\ri]^{-1}\frac{n+1}{8n} +
\frac{ (n+3) \cos^2 \delta - 4 }{4(1- \cos^2 \delta )}
$$
is strictly decreasing and has a unique zero $\tilde \delta_0$ in $(0,\frac\pi2)$. Hence, (i)
follows from \eqref{eq-compare-1}.

The proof of  (ii) is similar to that of (i) except we replace the lower bound $ \mu (\delta) > \tilde \mu (\delta) $
by a weaker lower bound  $\mu(\delta_0)>\frac{n}{(\sin\delta_0)^2}=\frac{n}{1-(\cos\delta_0)^2}$.

(iii)  follows from the fact
$$ \frac{n+1}{8n} +
\frac{ (n+3) \cos^2 \delta - 4 }{4(1- \cos^2 \delta )}  < 0. $$

\end{proof}

Theorem \ref{thm-RH rigidity 1} and Proposition \ref{prop-cosd-estimate} (iii) verify condition ({\bf b}) in the introduction.

\subsection{A Combined approach} \label{sec-combined}
It remains to   confirm the  case $  n = 3, 4 $ in condition ({\bf a}). 
To do so, we combine the  two  methods  leading to Theorem \ref{thm-method-1} 
and Theorem \ref{thm-RH rigidity 1}.

\begin{thm} \label{thm-combined}
Suppose $ 3 \le n \le 4$, Theorem \ref{thm-BM} is true on $ B(\delta) $ if
\be \label{eq-n345}
 \cos\delta >  \left( \frac{ 4 ( n+4) - 4 \sqrt{2n -1} }{ n^2 + 6 n + 17 } \right)^\frac12
 \approx
 \left\{
 \begin{array}{cl}
 0.6581, & \ n = 3 \\
 0.6130, & \ n = 4.
\end{array}
 \ri.
 \ee
\end{thm}

\begin{proof}
Let $ c = \cos \delta$.  \eqref{eq-n345} implies $ c^2 > \frac{1}{n+1}$. By
\eqref{eq-BM-est-method1-2-2}, we have
$ W(h) \ge Y(h) $
where
\begin{equation*}
\begin{split}
Y (h) = & \  \lf[ c + \frac{ (n+3) c^2 - 4  }{ 4 (1-c^2)}  \sqrt{2 (1-c^2) }  \ri]   
 \int_\Omega \lf(  \frac12 |h|^2 + \frac14 | \bD \th |^2 \ri) \vbg \\
&  +  \lf[ \frac12 + \frac{ (n+3)  c^2 - 4 }{4(1-c^2)} \ri] \int_\Omega \l \th^2 \vbg 
 + \frac{c}{4} \int_\Omega     | \bD h |^2 \ \vbg . 
\end{split}
\end{equation*}
Now \eqref{eq-n345} implies  \eqref{eq-lmaii-1}, i.e.
\be \label{eq-prior-c}
  c + \frac{(n+3)c^2 - 4 }{ 4 ( 1 - c^2)} \sqrt{ 2 ( 1 - c^2) } > 0 . 
\ee
To continue, we only need to assume  $  \frac12 +  \frac{ (n+3)c^2 - 4 }{4(1-c^2)}  < 0  $. 
(If $ n \ge 5$, this term would  be nonnegative by \eqref{eq-nge5}.)

Given any  constants $ \theta, \tau \in (0,1)$,
using the fact  $|\bD h|^2\ge \frac{1}{n} |\bD\th|^2$, $|h|^2\ge \frac{1}{n} \th^2$, $  \l \le 1$ and applying \eqref{eq-mu-new} as in Theorem \ref{thm-RH rigidity 1}, 
we have
\be\label{eq-mixed-4}
\begin{split}
Y (h) \ge &  \int_\Omega  \left\{ \frac{\theta c}{4} | \bD h |^2 + \frac14\lf[\frac{1-\theta}{n} c + c +\frac{(n+3)c^2-4}{2\sqrt{2(1-c^2)}}\ri]   | \bnabla \th  |^2 \ri.  \\
& + {\tau} \lf[  c+\frac{ (n+3)c^2-4}{2 \sqrt{2(1-c^2)}}\ri]  \frac{|h|^2}{2}
+ \frac{1-\tau}{n}\lf[   c+\frac{ (n+3)c^2-4}{2 \sqrt{2(1-c^2)}} \ri]  \frac{\th ^2}{2} \\
&  \lf.  +\lf[ 1 +  \frac{ (n+3)c^2 - 4 }{2(1-c^2)}   \ri]  \frac{ \th ^2}{2} \ri\}  \ \vbg   \\
\ge & \  \ep \lf(   \int_\Omega | \bD h |^2 +   |h|^2   \vbg  \ri) \\
& +  \bigg\{\frac12\lf[\frac{ (n+1) - \theta} {n} c+\frac{(n+3)c^2-4}{2\sqrt{2(1-c^2)}} \ri]\mu(\delta)+\frac{1-\tau}{n}\lf[   c+\frac{ (n+3)c^2-4}{2 \sqrt{2(1-c^2)}} \ri]  \\
&+\lf[ 1 +  \frac{ (n+3)c^2 - 4 }{2(1-c^2)}   \ri]  \bigg\}  \left( \int_\Omega \frac{\th^2}{2} \ \vbg \right)
 +E(h)
\end{split}
\ee
 where $ \ep = \min \lf\{ \frac{\theta c}{4} , \frac{\tau}{2} \lf[  c+\frac{ (n+3)c^2-4}{2 \sqrt{2(1-c^2)}}\ri]  \ri\} > 0 $,
  $ \mu(\delta)$ is the first nonzero Neumann eigenvalue of $B(\delta)$, and $E(h)$ is an error term satisfying
 \begin{equation*}
\begin{split}
| E  (h) | \le &
 C \lf[\int_\Omega (R(g)-R(\bg))\ \vbg+
2 \int_\S (H(g)-H(\bg))\ \vsg\ri]^2\\
& +  C \lf[ \int_\Omega \lf(|h|^2+|\bD h|^2\ri) \vbg
  + \int_\Sigma  ( |h|^2 + | h | |\bD h| )  \vsg  \ri]^2
  \end{split}
  \end{equation*}
with $C$ depending only on $B(\delta)$.

Apply  the eigenvalue estimate  $\mu(\delta) > \frac{n}{ (\sin \delta)^2} =  \frac{n}{1-c^2}$ (Lemma \ref{lma-mu-estimate} (ii)),    one checks (using {\em Mathematica}) that
\begin{equation} \label{eq-Mathematica}
\begin{split}
0 < & \ \frac12\lf[\frac{ n+1} {n} c+\frac{(n+3)c^2-4}{2\sqrt{2(1-c^2)}} \ri]\mu(\delta)+\frac{1}{n}\lf[   c+\frac{ (n+3)c^2-4}{2 \sqrt{2(1-c^2)}} \ri] \\
& \ +\lf[ 1 +  \frac{ (n+3)c^2 - 4 }{2(1-c^2)}   \ri]
\end{split}
\end{equation}
for $ 1 > c > 0.6378 $ when $ n =3 $ and
for $ 1 > c  >  0.5933$ when $ n =4$.
In particular,    \eqref{eq-Mathematica} is guaranteed by  \eqref{eq-n345}.

Therefore, there exist small positive constants $ \theta$, $\tau$ such that the coefficient of
$ \int_\Omega \frac{\th^2}{2} \vbg $  in \eqref{eq-mixed-4} is  positive. For these $ \theta $ and $ \tau$,
we  have
$$
W (h) \ge Y (h) \ge  \ep \lf(   \int_\Omega | \bD h |^2 +   |h|^2   \vbg  \ri)  + E(h).
$$
Arguing as in the proof of Theorem \ref{thm-RH rigidity 1} (the part following \eqref{eq-BM-est-method2-3}),
we conclude that Theorem \ref{thm-BM} holds on such a $B(\delta)$.
\end{proof}

\end{document}